\documentclass[12pt]{amsart}

\usepackage{amssymb,amsthm,amsmath}

\usepackage{mathtools}

\RequirePackage[dvipsnames,usenames]{color}
\input{kmacros3.sty}
\usepackage{mabliautoref}

\usepackage{tikz}
\usepackage{tikz-cd}
\usetikzlibrary{cd}
\usepackage{graphicx}
\usepackage[all,cmtip]{xy}

\numberwithin{theorem}{section}
\numberwithin{equation}{theorem}

\usepackage{bm}
\usepackage{pifont}
\usepackage{upgreek}

\usepackage{eucal}
\usepackage{ulem}
\usepackage{stmaryrd}

\numberwithin{equation}{theorem}

\usepackage{fullpage}

\usepackage{enumerate}
\usepackage{calc}

\usepackage{verbatim}
\usepackage{alltt}
\usepackage{scalerel,stackengine}

\def\todo#1{\textcolor{red}%
{\footnotesize\newline{\color{red}\fbox{\parbox{\textwidth-15pt}{\textbf{todo: } #1}}}\newline}}

\stackMath
\newcommand\reallywidehat[1]{%
\savestack{\tmpbox}{\stretchto{%
  \scaleto{%
    \scalerel*[\widthof{\ensuremath{#1}}]{\kern.1pt\mathchar"0362\kern.1pt}%
    {\rule{0ex}{\textheight}}
  }{\textheight}%
}{2.4ex}}%
\stackon[-6.9pt]{#1}{\tmpbox}%
}
\numberwithin{theorem}{section}

\begin{document}

\title{The Brian\c{c}on-Skoda Theorem via weak functoriality of big Cohen-Macaulay algebras}
\author{Sandra Rodr\'iguez-Villalobos}
\address{Department of Mathematics, University of Utah, Salt Lake City, UT 84112, USA}
\email{rodriguez@math.utah.edu}
\author{Karl Schwede}
\address{Department of Mathematics, University of Utah, Salt Lake City, UT 84112, USA}
\email{schwede@math.utah.edu}

\begin{abstract}
    We prove, given a sufficiently functorial assignment from rings to big Cohen-Macaulay algebras $R \mapsto B$, that the associated big Cohen-Macaulay closure operation on ideals $I \mapsto I B \cap R$ necessarily satisfies the Brian\c{c}on-Skoda type property.  The proof combines arguments of Lipman-Teissier, Hochster, Ma, and Hochster-Huneke.   Specializing to mixed characteristic, and utilizing a result of Bhatt on absolute integral closures, this recovers a slight strengthening of a result of Heitmann.
\end{abstract}
\maketitle

\section{Introduction}

The Brian\c{c}on-Skoda theorem relates integral closures of powers of ideals with ordinary powers of ideals
\[
    \overline{J^{n+\lambda-1}} \subseteq J^{\lambda}
\]
for integers $\lambda \geq 1$ where $J = (f_1, \dots, f_n)$ is generated by $n$ elements.  This was originally shown when the ambient ring is $\bC[x_1, \dots, x_d]$ in \cite{BrianconSkoda}, and an algebraic proof was given in \cite{LipmanSatheyeJacobianIdealsAndATheoremOfBS}.  There have since been many generalizations to classes of Noetherian rings with mild singularities.  Indeed, Lipman and Teissier showed that for a $d$-dimensional pseudo-rational Noetherian local ring, we have that 
\[
    \overline{J^{d+\lambda-1}} \subseteq J^{\lambda},
\]
\cite{LipmanTeissierPseudoRational}.  Under moderate hypotheses,\footnote{infinite residue field} every ideal can be generated by at most $d = \dim R$ elements up to integral closure \cite[Proposition 8.3.7]{HunekeSwansonIntegralClosure} (by  taking a minimal reduction) so that the Lipman-Teissier bound is not so dissimilar to the one from \cite{BrianconSkoda}, also see \cite[Corollary 2.2]{LipmanTeissierPseudoRational}.

While one does not expect that such results hold in arbitrarily singular rings, there have been numerous generalizations up to a closure operation on the term $J^{\lambda}$ (as well as other sorts of generalizations, see for instance \cite{LipmanSatheyeJacobianIdealsAndATheoremOfBS,HunekeUniformBounds,LipmanAdjointsAndPolarsOfIdeals,LazarsfeldPositivity2}).  The most notable example of the use of closure operations in this way is \cite{HochsterHunekeTC1}, where Hochster and Huneke showed that 
\[
    \overline{J^{n+\lambda-1}} \subseteq (J^{\lambda})^*
\]
for any Noetherian domain of characteristic $p > 0$.  Here $(-)^*$ denotes tight closure.  Later, they deduced the result for $+$-closure (still in characteristic $p > 0$): 
\[
    \overline{J^{n+\lambda-1}} \subseteq (J^{\lambda})R^+ \cap R =: (J^{\lambda})^+
\]
where $R^+$ is the absolute integral closure of $R$, see \cite[Theorem 7.1]{HochsterHunekeApplicationsofBigCM}.  These results have been generalized to characteristic zero for many closure operations via comparison with (methods related to) tight closure or plus closure in characteristic $p > 0$, see for example \cite[Theorem 1.3.7]{HochsterHunekeTightClosureInEqualCharactersticZero}, \cite[Proposition 8.6]{BrennerRescueSolid},  \cite[Theorem 6.13]{AschenbrennerSchoutensLefschetzExtensionsBigCM}.  

In mixed characteristic, we have such results for certain extensions of plus closure (extended plus and rank 1 closures) and even have variants of it for $+$-closure itself \cite{HeitmannPlusClosureMixedChar,HeitmannExtensionsOfPlusClosure}, also see \cite[Theorem 3.4]{HochsterVelezDiamondClosure}.  Indeed, the centrality of the Brian\c{c}on-Skoda theorem for closure operations has lead to Murayama coining as an axiom of potential closure operations \cite[Axiom 3.7]{MurayamaSymbolicTestIdeal}.

Big Cohen-Macaulay algebras themselves yield well behaved closure operations on ideals.  Indeed, if $B$ is a big Cohen-Macaulay $R$-algebra, then we can define the $B$-closure of any ideal $I \subseteq R$ to simply be $IB \cap R$ (such a closure defined by such an extension-contraction is called an \emph{algebra closure}).  For instance, under moderate hypotheses in characteristic $p > 0$, $B$-closure for  sufficiently large $B$ agrees with tight closure \cite[Theorem 11.1]{HochsterSolidClosure}.  For a comparison of closure operations inducing and coming from big Cohen-Macaulay algebras see \cite{DietzACharacterizationOfClosureOpsInducingBCM,RGClosureOperationsThatInducedBCM}.

The purpose of this article is to show that any closure operation defined by sufficiently functorial choices of big Cohen-Macaulay algebras automatically satisfies the Brian\c{c}on-Skoda theorem.

\begin{mainthm*}[{\autoref{MainTheorem.FullVersion}}]
    Suppose $(R, \fram)$ is a local excellent Noetherian domain and we are given a weakly functorial\footnote{In fact, we only need weak functoriality for $R$-algebra surjections, see \autoref{def.WeaklyFunctorialBCMAssignments}.} assignment from essentially of finite type local integral domain $R$-algebras $(S, \fran) \supseteq (R, \fram)$ to balanced big Cohen-Macaulay algebras, $S \mapsto B_S$.  Suppose $J \subseteq R$ is an ideal that can be generated by $n$ elements.

    Then for any integer $\lambda \geq 1$:
    \[ 
        \overline{J^{n + \lambda-1}} \subseteq J^{\lambda} B_R \cap R.
    \]
\end{mainthm*}

For parameter ideals, we can easily explain exactly what functoriality we need.

\begin{theoremA*}[{\autoref{MainTheorem.ParameterIdealCaseGeneralized}}]
        Suppose $(R, \fram)$ is a normal local excellent domain and $J \subseteq R$ is an ideal generated by a partial system of parameters $J = (f_1, \dots, f_n)$.  Let $S = R[Jt]^{\nm}$ denote the normalized Rees algebra with homogeneous maximal ideal $\fran = \fram S + S_{> 0}$.  Suppose we have a commutative diagram:
    \[
        \xymatrix{
            S_{\frn} \ar@{->>}[r] \ar[d] & R \ar[d] \\
            B \ar[r] & C.
        }
    \]
    where the top horizontal map is projection onto the degree 0 part and where $B$ and $C$ are balanced big Cohen-Macaulay $S_{\frn}$ and $R$-algebras respectively.  Then for any integer $\lambda \geq 1$:
    \[ 
        \overline{J^{n + \lambda-1}} \subseteq J^{\lambda} C \cap R.
    \]
\end{theoremA*}
The normal hypothesis on $R$ can be weakened, see the referenced theorem.  Note, if for instance $S_{\frn}$ is Cohen-Macaulay, then we take $C$ to be \emph{any} balanced big Cohen-Macaulay $R$-algebra.

Regardless, in mixed characteristic, by utilizing \cite{BhattAbsoluteIntegralClosure}, \cite[Corollary 2.10]{BMPSTWW-MMP} to create our weakly functorial big Cohen-Macaulay algebras, our main result specializes to a partial strengthening of some of the main results of \cite{HeitmannPlusClosureMixedChar,HeitmannExtensionsOfPlusClosure}.  

\begin{cor*}[{\autoref{cor.HeitmannExtension}}]
    Suppose $(R, \fram)$ is an excellent Noetherian local domain of mixed characteristic $(0, p> 0)$.  Let $\widehat{R^+}$ denote the $p$-adic completion of the absolute integral closure of $R$.  Then for any ideal $J \subseteq R$ generated by $n$ elements, and any integer $\lambda \geq 1$
    \[
        \overline{J^{n + \lambda-1}} \subseteq J^{\lambda} \widehat{R^+} \cap R.
    \]
\end{cor*}

\subsection{An outline of the proof}

We first prove our main result in the case that $\lambda = 1$ and $J$ is a parameter ideal.  For this we begin similarly to the method of Lipman-Teissier (\autoref{lem.LipmanTeissierWhichCohomologySentToZero}) which we then combine with an argument due to Ma in (\autoref{MainTheorem.ParameterIdealCase}) using a Sancho de Salas sequence.  To handle the case when $\lambda \geq 1$ we use an argument due to Hochster as presented in Lipman-Teissier, see \autoref{MainTheorem.ParameterIdealCaseGeneralized}.  To generalize to the case of non-parameter ideals, we mimic an argument of Hochster-Huneke which does a generic computation, see \autoref{MainTheorem.FullVersion}.

\subsection*{Acknowledgements}
Sandra Rodr\'iguez-Villalobos was supported by  NSF Grant DMS-2101800.  Karl Schwede was supported by NSF Grant \#2101800 and NSF FRG Grant DMS-1952522.  This material based upon work supported by the National Science Foundation under Grant No. DMS-1928930, while the authors were in residence at the Simons Laufer Mathematical Sciences Institute (formerly MSRI) in Berkeley, California, during the Spring 2024 semester.  The authors thank Linquan Ma, Ilya Smirnov, Shunsuke Takagi, and Kevin Tucker for valuable conversations.  The authors also thank Ray Heitmann, Linquan Ma, Irena Swanson and the referee for valuable comments on previous drafts.

\section{Background and preliminaries}

We begin by recalling Hochster's notion of big Cohen-Macaulay algebras.  See \cite{HochsterTopicsInTheHomologicalTheory,SharpCMProperttiesForBalancedBCM}, also compare with \cite[Section 2.1]{BhattAbsoluteIntegralClosure} for an essentially equivalent notion defined by local cohomology.

\begin{definition}[BCM algebras]
    Suppose $(R, \fram)$ is a Noetherian local ring.  A \emph{balanced big Cohen-Macaulay} $R$-algebra (or $R$-module) is an $R$-algebra $B$ (respectively an $R$-module) such that every system of parameters on $R$ is a regular sequence on $B$.\footnote{If at least one system of parameters becomes a regular sequence, the \emph{balanced} modifier is removed, but if one is willing to $\fram$-adically complete such a $B$, it becomes balanced \cite[Exercise 8.1.7 \& Theorem 8.5.1]{BrunsHerzog}.}  Such $B$ we call \emph{BCM} (note the balanced is implicit but suppressed).
\end{definition}

It is not obvious that any local ring has a balanced BCM module or algebra.  However, BCM modules and algebras exist quite generally.  In equal characteristic $p > 0$ or $0$, see for example \cite{HochsterTopicsInTheHomologicalTheory,HochsterHunekeInfiniteIntegralExtensionsAndBigCM,HochsterHunekeApplicationsofBigCM,SchoutensCanonicalBCMAlgebrasAndRational,HunekeLyubeznikAbsoluteIntegralClosure,MurayamaSymbolicTestIdeal}.  In mixed characteristic see for instance \cite{AndreWeaklyFunctorialBigCM,GabberMSRINotes,HeitmannMaBigCohenMacaulayAlgebraVanishingofTor,BhattAbsoluteIntegralClosure}.  

We will need the following well known result whose proof we include because we do not know of a suitable reference.  A different argument which can also be used to obtain the same result can be found in \cite{KovacsSerresConditionVanishingMathOverlow}.

\begin{lemma}
    \label{lem.GradeVanishingFact}
    Let $R$ be a Noetherian local ring.  
    Suppose $B$ is a BCM $R$-module.  If $I$ is an ideal of height\footnote{the minimum height of an associated prime of $I$} $r$, then $I$ contains a partial system of parameters of length $r$ and so
    \[
        H^i_I(B) = 0 \;\;\; \text{ for $i = 0, \dots, r-1$}.
    \]
\end{lemma}
\begin{proof}
    First we produce our sequence.  Let $s$ be the maximal length of a partial system of parameters, $x_1, \dots, x_s$, contained inside $I$.  We claim that $s = r$.  Suppose for a contradiction that $s < r$.  As every minimal prime $Q_1, \dots, Q_t$ of $(x_1, \dots, x_s)$ has height $\leq s$ (by Krull's height theorem), none of them can contain $I$.  Thus by prime avoidance there exists $x_{s+1} \in I$ not contained in any $Q_i$.  As $x_1, \dots, x_s$ is part of a system of parameters, $\dim R/(x_1, \dots, x_s) = \dim R - s$.  Thus, since $\overline{x_{s+1}}$ is not in any minimal prime of $R/(x_1, \dots, x_s)$, $\dim R/(x_1, \dots, x_{s+1}) = \dim R - s - 1$ and so $x_1, \dots, x_{s+1}$ is part of a system of parameters, a contradiction.

    Setting $x_1, \dots, x_r$ to be the guaranteed partial system of parameters, the vanishing follows by induction on $r$ (since $\overline{x_2}, \dots, \overline{x_r}$ is a system of parameters of $R/(x_1)$ and $B/x_1 B$ is BCM over $R/(x_1)$) and the long exact sequence 
    \[ 
        \dots \to H^{i-1}_I(B/x_1 B) \to H^i_I(B) \xrightarrow{\cdot x_1} H^i_I(B) \to H^i_I(B/x_1 B) \to \dots
    \] 
    since the map labeled $\cdot x_1$ cannot be injective unless $H^i_I(B) = 0$.
\end{proof}

\begin{definition}
    Suppose $R$ is a Noetherian domain.  For an ideal $I \subseteq R$, let $\widetilde{I}$ denote the integral closure of $IR^{\mathrm{N}}$  where $R^{\mathrm{N}}$ is the normalization of $R$.
\end{definition}

\begin{definition}[Special functoriality of Big Cohen-Macaulay algebras]
    Suppose $(R, \fram)$ is a local Noetherian domain and $J \subseteq R$ is an ideal.  Set $S = R \oplus \widetilde{J}t \oplus \widetilde{J^2}t^2 \oplus \dots$, assume it is Noetherian, and note we have a surjection $\pi : S \to R$ by projection onto degree zero.  Set $\frn = \fram S + S_{>0}$ (a maximal ideal of $S$ mapping to $\fram$).  We say that a BCM $R$-algebra $C$ is \emph{Rees-$J$ functorial} if there exists a BCM $S_{\frn}$-algebra $B$ so that the following diagram commutes:
    \[
        \xymatrix{
            S_{\frn} \ar[d] \ar@{->>}[r] & R \ar[d] \\
            B \ar[r] & C.
        }
    \]    
\end{definition}

\begin{definition}[Weakly functorial Big Cohen-Macaulay assignments]
    \label{def.WeaklyFunctorialBCMAssignments}
    Suppose we have a subcategory $\sC$ of Noetherian rings (not necessarily full).  We say an assignment for each ring $R \in \sC$ to a BCM $R$-algebra $B_R$ is \emph{weakly functorial} if for each map $f : S \to R$ in $\sC$, there exists some ring map $g : B_S \to B_R$ such that the following diagram commutes:
    \[
        \xymatrix{
            S \ar[d] \ar[r]^f & R \ar[d] \\
            B_S \ar[r]_g & B_R
        }
    \]

    The usual category we will need to work on is as follows.  Fix $(R, \fram)$ to be an excellent Noetherian local domain.  Consider the category of local $R$-algebras $(S, \fran)$ where $S$ is a local domain essentially of finite type over $R$, and the structural map $R \to S$ is local and injective.  The maps in this category are surjective $R$-algebra homomorphisms.  We denote this category:
    \begin{equation}
        \label{eq.OurCat.WeaklyFunctorialBCMAssignments}
        \sD_R.
    \end{equation}
    In the end, we can even restrict our category to a certain tree of finitely many surjective maps, see \autoref{rem.Functoriality}.
\end{definition}

In particular, the assignment of a Noetherian domain $R$ to $R^+$ in positive characteristic is weakly functorial \cite{HochsterHunekeInfiniteIntegralExtensionsAndBigCM}, and the assignment to the $p$-adic completion $\widehat{R^+}$ (which is BCM by \cite{BhattAbsoluteIntegralClosure}, \cite[Corollary 2.10]{BMPSTWW-MMP}) in mixed characteristic is weakly functorial (also see \cite{AndreWeaklyFunctorialBigCM}).  In characteristic zero, there are different approaches for certain maps of (certain) local rings and various sorts of weak functoriality, see for instance 
\cite{HochsterHunekeInfiniteIntegralExtensionsAndBigCM}, \cite[Theorem 2.3]{HochsterHunekeApplicationsofBigCM}, \cite{SchoutensCanonicalBCMAlgebrasAndRational}, \cite{DietzRGBCMViaUltraEqualChar0}, or \cite[Theorem 2.8]{MurayamaSymbolicTestIdeal}.  Some weak functoriality in characteristic zero can also be obtained by reducing to mixed characteristic and then inverting $p > 0$.


We record the following variant of the Sancho de Salas sequence for future reference.  It is a special case of a result found in \cite[Equation (SS), Page 150]{LipmanCohenMacaulaynessInGradedAlgebras}, see also \cite{HeitmannMaBigCohenMacaulayAlgebraVanishingofTor} as pointed out by the referee.  

\begin{proposition}[{\cite{SanchodeSalasBlowingupmorphismswithCohenMacaulayassociatedgradedrings,LipmanCohenMacaulaynessInGradedAlgebras,HeitmannMaBigCohenMacaulayAlgebraVanishingofTor}}] 
    \label{SanchoDeSalasNonMaximalVariant}
    Let $R$ be an excellent Noetherian domain.  Fix $J = (f_1, \dots, f_n)$ to be an ideal.  Set $\widetilde{J^m} = \overline{J^m R^{\mathrm{N}}}$ where $R^{\mathrm{N}}$ is the normalization of $R$.  Let $T = R \oplus \widetilde{J}t \oplus \widetilde{J^2}t^2 \oplus \dots$ denote the partially normalized Rees algebra.  Let $\pi : X = \Proj T \to \Spec R$.  Let $E \subseteq X$ denote the inverse image of $V(J)$.

    Then 
    \[
        [H^{n}_{JT+(Jt)T}(T)]_0 \to H^n_J(R) \to H^n_E(X, \cO_X) 
    \]
    is exact.
\end{proposition}
\begin{proof}
    We need to translate our statement into the language of \cite{LipmanCohenMacaulaynessInGradedAlgebras}.  Notice first that $H^i_{JT + (Jt)T}(T) = H^i_{JT + T_{>0}}(T)$ as the ideals define the same sets.  Hence it suffices to show that $[H^{n}_{JT+T_{>0}}(T)]_0 \to H^n_J(R) \to H^n_E(X, \cO_X)$ is exact.  But this is what was shown in \cite[(SS)]{LipmanCohenMacaulaynessInGradedAlgebras} if one sets $i = n$, $\fram = J$, $G = N = T$, and $P = T_{>0}$ (notice that the ideal $\fram$ in Lipman's notation need not be maximal).  
\end{proof}

\section{The main result for parameter ideals}

We begin by recalling the following result of Lipman-Teissier, found within the proof of \cite[Theorem 2.1]{LipmanTeissierPseudoRational}.

\begin{lemma}[Lipman-Teissier]
    \label{lem.LipmanTeissierWhichCohomologySentToZero}
    Suppose that $R$ is a Noetherian domain, $J = (f_1, \dots, f_n)$ and $\pi : X \to \Spec R$ is the blowup of $J$ and set $E = \pi^{-1}(V(J))$.  Fix $h \in \Gamma(X, J^n \cO_X) \cap R$.  Then the \Cech{} class $\big[ {h \over f_1 \dots f_n}\big]$ is in the kernel of the map  
    \[
        H^n_{J}(R) \to H^n_E(X, \cO_X) = \myH^n \myR \Gamma_{J}(\myR \Gamma(X, \cO_X)).
    \]
    As a consequence, if $Y$ is (or factors through) the normalized blowup of $J$, then for any $h \in \overline{J^n}$, we have that $[ {h \over f_1 \dots f_n}\big]$ is in the kernel of
    \[
        H^n_{J}(R) \to \myH^n \myR \Gamma_{J}(\myR \Gamma(Y, \cO_Y)).
    \]  
\end{lemma}
\begin{proof}
    We recall the argument of Lipman-Teissier for the convenience of the reader.  
    We may assume $J \neq 0$.
    If $n = 1$ then $X = \Spec R$, the map $H^1_J(R) \to H^1_E(X, \cO_X)$ is the identity, and $H^1_J(R) = R[1/f_1]/R$.  If $h = xf_1 \in (f_1)$, then $[h/f_1] = [x/1]$ is already zero and we are done, so we may assume that $n \geq 2$.  
    Consider the standard blowup charts $X_i \subseteq X$ defined by the property that $f_i \cO_{X_i} = J \cO_{X_i}$.  Set $U = X \setminus E$ and notice that the $U_i := X_i \cap (X \setminus E)$ are identified with $\Spec R[1/f_i]$.  As $X \to \Spec R$ restricts to an isomorphism $U \to \Spec R \setminus V(J) =: W$, $H^{n-1}(U, \cO_X)$ is identified with $H^{n-1}(W, \cO_{\Spec R})$ which is isomorphic to $H^n_J(R)$ as $n \geq 2$.
    With this identification in mind, it suffices to show that $\big[ {h \over f_1 \cdots f_n}\big] \in H^{n-1}(U, \cO_X)$ maps to zero in $H^n_E(X, \cO_X)$.  
    
    Due to the exact sequence $H^{n-1}(X, \cO_X) \to H^{n-1}(U, \cO_X) \to H^n_E(X, \cO_X)$ we must show that $\big[ {h \over f_1 \dots f_n}\big]$ is the image of a class in $H^{n-1}(X, \cO_X)$.  We do this via \Cech cohomology aligning the cover $X_i$ of $X$ with the cover $U_i$ of $U$.  Set $V = \bigcap_i X_i$ and notice that since $f_i \cO_{X_i} = J\cO_{X_i}$, we have that $J^n \cdot \cO_V = (f_1 \cdots f_n) \cO_V$.  Thus considering $h \in \Gamma(V, J^n \cO_{X}) = \Gamma(V, (f_1 \cdots f_n) \cO_X)$, we see that $\big[ {h\over f_1 \cdots f_n}\big]$ exists as a \Cech class in 
    \[ 
        H^{n-1}(X, \cO_X) = \coker \Big(\prod_j \Gamma(\hat{X_j}, \cO_X) \to \Gamma(V, \cO_X) \Big) 
    \] 
    where $\hat{X_j} = \bigcap_{i \neq j} X_i$.  The first statement follows.

    For the second statement, we recall that $\Gamma(Y, J^n \cO_Y) \cap R = \overline{J^n}$.  The result follows.
\end{proof}


We continue to follow the argument of Lipman-Teissier \cite{LipmanTeissierPseudoRational}, now combining with an argument of Ma's taken from \cite[Section 3]{MaThevanishingconjectureformapsofTorandderivedsplinters}, \cite[Corollary 4.3]{HeitmannMaBigCohenMacaulayAlgebraVanishingofTor} or \cite[Proposition 5.11]{MaSchwedeSingularitiesMixedCharBCM}.

\begin{theorem}\label{MainTheorem.ParameterIdealCase}
    Suppose $(R, \fram)$ is an excellent local domain with normalization $R^{\mathrm{N}}$ and $f_1, \dots, f_n$ is part of a system of parameters which generates an ideal $J$.
    Then 
    \[
        \overline{(f_1, \dots, f_n)^{n}}  \subseteq (f_1, \dots, f_n) C
    \]
    where $C$ is any Rees-$J$ BCM $R$-algebra.
\end{theorem}
\begin{proof}
    Fix $h \in \overline{(f_1, \dots, f_n)^{n}}$.  
    Again we write $\widetilde{J^m} = \overline{J^m R^{\mathrm{N}}}$.  Let $S = R \oplus \widetilde{J}t \oplus \widetilde{J^2}t^2 \oplus \dots$ denote the partially normalized Rees algebra (fully normalized if $R$ is normal) so that $\pi : X = \Proj S \to \Spec R$ is the normalized blowup with $E = \pi^{-1}(V(J))$.  
    
    Set $J' = JS + S_{>0} \subset S$.  Since $\sqrt{(Jt)S} = S_{>0}$, we see that $JS + S_{>0}$ has the same radical as $JS + (Jt)S$ and so we can use them interchangeably when computing their local cohomology.   

    
    By our variant of the Sancho de Salas sequence \autoref{SanchoDeSalasNonMaximalVariant}:
    \[
        [H^n_{J'}(S)]_0 \to H^n_J(R) \to H^n_E(X, \cO_X)
    \]
    is exact and hence 
    \[ 
        \ker \big( H^n_J(R) \to H^n_E(X, \cO_X) \big) \subseteq \Image\big( H^n_{J'}(S_{\frn}) \xrightarrow{\psi} H^n_J(R) \big)
    \]
    where $\frn = \fram S + S_{>0}$.
    Therefore, by \autoref{lem.LipmanTeissierWhichCohomologySentToZero}, since $h \in \overline{J^n}$ we see that $[h/(f_1 \cdots f_n)] \in \Image \psi$. 
    
    By the hypothesis about $C$, there exists a commutative diagram
    \[
        \xymatrix{
            S_{\frn} \ar[d] \ar@{->>}[r] & R \ar[d] \\
            B \ar[r] & C
        }
    \]
    where $B$ and $C$ are BCM $S_{\frn}$ and $R$-algebras, respectively.  Notice the image of $J'$ in $R$ is $J R$.  We also notice that the height of $J'$ is $n+1$ (as the height of the parameter ideal $J$ was $n$).  We take local cohomology and obtain the diagram:
    \[
        \xymatrix{
            H^{n}_{J'}(S_{\frn}) \ar[d] \ar[r] & H^n_{J}(R) \ar[d] \\
            H^n_{J'}(B) \ar[r] & H^n_J(C)
        }
    \]
    By \autoref{lem.GradeVanishingFact} and our observation on the height of $J'$, we see that $H^n_{J'}(B) = 0$.  Thus, since $[h/(f_1 \cdots f_n)] \in \Image \psi$, we see that 
    \[ 
        [h/(f_1 \cdots f_n)] \mapsto 0 \in H^n_J(C).
    \]

    Since $C$ is BCM, and $J$ is a parameter ideal, we have that $C/JC \to H^n_J(C)$ injects (this follows as $H^n_J(C) = \varinjlim_i C/(x_1^i, \cdots, x_n^i)C$ and the maps in the system, $\times (x_1 \cdots x_n)$, inject due to \autoref{lemmaColonRegularSequence} immediately below), and hence $h \in JC$, as desired.
\end{proof}

\subsection{Generalizing beyond {$\lambda = 1$}}

To generalize \autoref{MainTheorem.ParameterIdealCase} beyond the case $\lambda = 1$, we follow an argument due to Hochster, found in 
\cite[Section 3]{LipmanTeissierPseudoRational}, for which we need our regular sequences on our big Cohen-Macaulay algebra to be permutable (that is, we explicitly use that it is balanced).

\begin{definition}
	We say that $f_1, \ldots, f_n$ is a permutable regular sequence if  $f_{\sigma(1)}, \ldots, f_{\sigma(n)}$ is a regular sequence for every $\sigma\in S_n$.
\end{definition}

The following well-known result will play a central role in reducing to the case $\lambda = 1$.
\begin{lemma} \label{lemmaColonRegularSequence}
	If $f_1,\ldots, f_n$ is a permutable regular sequence on a possibly non-Noetherian ring $R$, then 
	$$((f_1^{\lambda_1},\ldots, f_n^{\lambda_n})R: f_1^{\mu_1}\cdots f_n^{\mu_n})=(f_1^{\lambda_1-\mu_1},\ldots, f_n^{\lambda_n-\mu_n})R.$$
\end{lemma}

Using the previous lemma, we can describe the powers of $I$ as intersections of ideals generated by powers of the elements of the generating regular sequence.

\begin{lemma}[Hochster, \cf {\cite[Section 3]{LipmanTeissierPseudoRational}}]\label{LemmaIntersectionsOfParameterIdeals}
	If $I=(f_1,\ldots, f_n)R$ with $f_1,\ldots, f_n$ a permutable regular sequence on a possibly non-Noetherian ring $R$, then
	$$I^\lambda=\bigcap\limits_{\lambda_1,\ldots, \lambda_n}(f_1^{\lambda_1},\ldots, f_n^{\lambda_n})R$$
	where $(\lambda_1,\ldots, \lambda_n)$ runs through all $n$-tuples of strictly positive integers with 
	$\lambda_1+\ldots+ \lambda_n=\lambda+n-1$.
\end{lemma}
\begin{proof}
    The containment $(\subseteq)$ follows from the pigeonhole principle and so we prove ($\supseteq$).

	We proceed by induction on $\lambda$.
	When $\lambda=1$ and $(\lambda_1,\ldots, \lambda_n)$ is an $n$-tuple of strictly positive integers with 
	$\lambda_1+\ldots+ \lambda_n=\lambda+n-1=n$, it follows that $\lambda_i=1$ for all $i$ so the statement reduces to 
	$$I=(f_1,\ldots, f_n)R.$$

	Suppose 
	$$I^{\lambda-1}=\bigcap\limits_{\lambda_1,\ldots, \lambda_n}(f_1^{\lambda_1},\ldots, f_n^{\lambda_n})R$$
	where $(\lambda_1,\ldots, \lambda_n)$ runs through all $n$-tuples of strictly positive integers with 
	$\lambda_1+\ldots+ \lambda_n=\lambda-1+n-1$.
	Let $x\in \bigcap\limits_{\gamma_1,\ldots, \gamma_n}(f_1^{\gamma_1},\ldots, f_n^{\gamma_n})R$ where $(\gamma_1,\ldots, \gamma_n)$ runs through all $n$-tuples of strictly positive integers with 
	$\gamma_1+\ldots+ \gamma_n=\lambda+n-1$.
	By the induction hypothesis, $x\in I^{\lambda-1}$, so
	$$x=\sum_{\substack{\alpha_1+\alpha_2+\ldots+ \alpha_n\\=\lambda-1}}a_{\alpha_1,\alpha_2,\ldots, \alpha_n}f_1^{\alpha_1}\cdots f_n^{\alpha_n}.$$
	Note that if $\beta_1+\beta_2+\ldots+ \beta_n=\lambda-1$, for $(\alpha_1,\alpha_2,\ldots, \alpha_n)\neq (\beta_1,\beta_2,\ldots, \beta_n)$, 
	we have that $a_{\alpha_1,\alpha_2,\ldots, \alpha_n}f_1^{\alpha_1}\cdots f_n^{\alpha_n}\in (f_1^{\beta_1+1},\ldots, f_n^{\beta_n+1})$
	since $\alpha_i>\beta_i$ for some $i$.
	Thus, $$a_{\beta_1,\beta_2,\ldots, \beta_n}f_1^{\beta_1}\cdots f_n^{\beta_n}\in (f_1^{\beta_1+1},\ldots, f_n^{\beta_n+1}).$$ 
	By \autoref{lemmaColonRegularSequence}, it follows that $a_{\beta_1,\beta_2,\ldots, \beta_n}\in I$ for all $(\beta_1,\beta_2,\ldots,\beta_n)$ with $\beta_1+\beta_2+\ldots+ \beta_n=\lambda-1$. As a consequence, we have that $x\in I^\lambda$.
\end{proof}

We can now extend our result beyond $\lambda=1$ just as in \cite{LipmanTeissierPseudoRational}.

\begin{theorem}\label{MainTheorem.ParameterIdealCaseGeneralized}
    Suppose $(R, \fram)$ is an excellent local domain and $f_1, \dots, f_n$ is a partial system of parameters generating an ideal $J$. 
    Then for any integer $\lambda > 0$
    \[
        \overline{(f_1, \dots, f_n)^{\lambda+n-1}}  \subseteq (f_1, \dots, f_n)^{\lambda} C
    \]
    where $C$ is a Rees-$J$ BCM $R$-algebra.
\end{theorem}
\begin{proof}
	Let $(\lambda_1,\ldots, \lambda_n)$  be an $n$-tuple of strictly positive integers with 
	$\lambda_1+\ldots+ \lambda_n=\lambda+n-1$, $\mu=\max\{\lambda_1,\ldots, \lambda_n\}$ and $h\in  \overline{(f_1, \dots, f_n)^{\lambda+n-1}}$. Then, we have that
	$$f_1^{\mu-\lambda_1}\cdots f_n^{\mu-\lambda_n}h\in J^{n\mu-\lambda-n+1}\overline{(f_1, \dots, f_n)^{\lambda+n-1}}\subseteq \overline{(f_1, \dots, f_n)^{n\mu}}.$$
	Since $\overline{(f_1, \dots, f_n)^{n\mu}}=\overline{(f_1^\mu, \dots, f_n^{\mu})^{n}}$ and since,
	by \autoref{MainTheorem.ParameterIdealCase},
	$\overline{(f_1^\mu, \dots, f_n^\mu)^{n}}  \subseteq (f_1^\mu, \dots, f_n^\mu) C$, we obtain that 
	$$h\in (f_1^\mu, \dots, f_n^\mu) C: f_1^{\mu-\lambda_1}\cdots f_n^{\mu-\lambda_n}.$$
	Using \autoref{lemmaColonRegularSequence}, it follows that $h\in (f_1^{\lambda_1}, \dots, f_n^{\lambda_n}) C$.
	Since $\lambda_1,\ldots, \lambda_n$ are arbitrary, by \autoref{LemmaIntersectionsOfParameterIdeals},
	$h\in (f_1,\ldots, f_n)^{\lambda}C$. Hence, we obtain 
	$$\overline{(f_1, \dots, f_n)^{\lambda+n-1}}\subseteq (f_1,\ldots, f_n)^{\lambda}C.$$
\end{proof}

\section{The general case, non-parameter ideals} 
We will now follow an argument of Hochster-Huneke \cite[Theorem 7.1]{HochsterHunekeApplicationsofBigCM} to extend the main result to ideals that are not necessarily generated by partial systems of parameters.  

\begin{theorem}\label{MainTheorem.FullVersion}
    Suppose $(R, \fram)$ is an excellent local domain and $f_1, \dots, f_n\in \fram$.  Suppose we have a weakly functorial BCM assignment on the category $\sD_R$ from \autoref{eq.OurCat.WeaklyFunctorialBCMAssignments}, $S \mapsto B_S$.
    Then 
    \[
        \overline{(f_1, \dots, f_n)^{\lambda+n-1}}  \subseteq (f_1, \dots, f_n)^{\lambda} B_R.
    \]
\end{theorem}
\begin{proof}
Let $c=\lambda+n-1$, $I=(f_1,\ldots,f_n)$ and suppose $h\in \overline{I^{c}}$. 
Then, $h^n+a_1h^{n-1}+a_2h^{n-2}+...+a_n=0$ for some $a_i\in (I^{c})^i$; that is,
$$h^n-\sum_{i=1}^n\left(\sum_{\substack{\alpha_1+\ldots+\alpha_n\\=ci}} a_{\alpha_1,\ldots, \alpha_n}f_1^{\alpha_1}\cdots f_n^{\alpha_n}\right)h^{n-i}=0$$
for some $a_{\alpha_1,\ldots, \alpha_n}\in R$.

For each $i\in\{1,\ldots, n\}$ and for each $\alpha_1+\ldots+\alpha_n=ci$, let $x_i, y_{\alpha_1,\ldots, \alpha_n},$ and $z$ be indeterminates.
Let $T=R[\underline{x}, \underline{y}, z]$ and consider the $R$-algebra map $\varphi:T\rightarrow R$ determined by
$\varphi(y_{\alpha_1,\ldots, \alpha_n})=a_{\alpha_1,\ldots, \alpha_n}$, $\varphi(x_i)=f_i$ and $\varphi(z)=h$. 
Then, if $$g=z^n-\sum_{i=1}^n\Big(\sum_{\substack{\alpha_1+\ldots+\alpha_n\\=ci}} y_{\alpha_1,\ldots, \alpha_n}x_1^{\alpha_1}\cdots x_n^{\alpha_n}\Big)z^{n-i},$$ 
$g\in \ker \varphi$ and so we have an induced map $\psi: T/(g)\rightarrow R$. 
Note that $g$ is linear as a polynomial on $y_{cn,0,\ldots, 0}$.
Additionally, the coefficient of $y_{cn,0,\ldots, 0}$ in $g$, which is $x_1^{cn}$, and the constant term of $g$ as a  polynomial in $y_{cn,0,\ldots, 0}$ are relatively prime. 
It follows that $g$ is irreducible and $T/(g)$ is a domain.  

Let $P=\psi^{-1}(\fram)$ and $S=(T/(g))_P$.  Observe that $x_1, \dots, x_n$ is a partial system of parameters on $S$ as $x_1, \dots, x_n, g$ is a partial system of parameters on $T$ localized at the inverse image of $\fram$.

We have a ring homomorphism $S\rightarrow R$ and $S$ is an excellent local domain.
Then there exists a commutative diagram:
   \[
        \xymatrix{
            S \ar[d] \ar@{->>}[r] & R \ar[d] \\
            B_S \ar[r] & B_R.
        }
    \]
Since $z\in\overline{(x_1,\ldots,x_n)^{c}}$ and $x_1,\ldots,x_n$ is a partial system of parameters, it follows from \autoref{MainTheorem.ParameterIdealCaseGeneralized} that $z\in (x_1,\ldots, x_n)^{\lambda}B_S$ and, taking images under the map $B_S\rightarrow B_R$, we obtain that  $z\in (f_1,\ldots, f_n)^{\lambda}B_R$.
\end{proof}

Note Hochster-Huneke carried out essentially 
the same argument as part of proving Brian\c{c}on-Skoda for $+$-closure in characteristic $p>0$.  In their argument they set $T = \bF_p[\underline{x}, \underline{y}, z]$, and while we could replace $\bF_p$ with $\bZ$, our choice keeps us in a similar class of rings (ie, equal characteristic $0$ or $p > 0$ stays equal characteristic) where numerous weak functoriality results have previously been shown.

\begin{remark}
    \label{rem.Functoriality}
    We can now explain somewhat more precisely what functoriality we need.  Indeed, fixing a $\lambda$, for each generator $h_i$ of $\overline{J^{n+\lambda -1}}$ we can associate $R_i = (T_i/(g_i))_{P_i}$ as in the proof above with a corresponding $z_i \in S_i$ mapping to $h_i$.  That $z_i$ is in the integral closure of an $n$-generated parameter ideal $J_i \subseteq R_i$ (mapping to $J = (f_1, \dots, f_n)$).  Hence using a Rees-$J_i$ BCM algebra for $R_i$ is sufficient as long as it maps weakly functorially to BCM $R$-algebra.  In particular, for a fixed $\lambda$, we need only a BCM algebra assignment for a certain finite tree of maps rooted at $R$ (the final object of the associated finite category).
\end{remark}

\section{An application and further thoughts}

Utilizing \cite{BhattAbsoluteIntegralClosure}, \cite[Corollary 2.10]{BMPSTWW-MMP},  which proves that the $p$-adic completion of the absolute integral closure is balanced big Cohen-Macaulay (and certainly sufficiently weakly functorial for our purposes), we obtain the following.

\begin{corollary}
    \label{cor.HeitmannExtension}
    Suppose $(R, \fram)$ is an excellent Noetherian local domain of mixed characteristic $(0, p> 0)$.  Let $\widehat{R^+}$ denote the $p$-adic completion of the absolute integral closure of $R$.  Then for any ideal $J \subseteq R$ generated by $n$ elements, and any integer $\lambda \geq 1$
    \[
        \overline{J^{n + \lambda-1}} \subseteq J^{\lambda} \widehat{R^+} \cap R.
    \]
\end{corollary}

\begin{remark}
This is very close to several results of Heitmann.  In \cite[Theorem 2.13]{HeitmannPlusClosureMixedChar}, he proved that if $J = (f_1, \dots, f_n)$ and $p \in \sqrt{(x_1,x_2)}$, then $\overline{J^{n+\lambda-1}} \subseteq J^{\lambda} R^+$ and so we recover that result in our setting of an \emph{excellent} ring.  The point being that since a power of $p$ is in $J$, then $J R^+ \cap R = J \widehat{R^+} \cap R$.  In \cite[Theorem 2.13]{HeitmannPlusClosureMixedChar}, Heitmann also proved that $\overline{J^{n + \lambda}} \subseteq J^{\lambda} R^+ \cap R$ in general.

On the other hand, in \cite[Theorem 4.2]{HeitmannExtensionsOfPlusClosure}, Heitmann proved that $\overline{J^{n+\lambda-1}} \subseteq (J^{\lambda})^{\mathrm{epf}}$.
Because of \cite{BrennerMonsky}, we might expect that $J^{\mathrm{epf}} \supsetneq J \widehat{R^+} \cap R$ in general.  It seems possible that $J^{\mathrm{epf}}$ and $J \widehat{R^+} \cap R$ agree for parameter ideals $J$ based on the characteristic $p > 0$ picture, see \cite{SmithTightClosureParameter}.
\end{remark}

\begin{remark}[Comparison with multiplier/test ideal Skoda theorems]
    Another way to obtain Brian\c{c}on-Skoda-type theorems is to use the Skoda-type theorems for multiplier or test ideals.  

    For example, if $R$ is essentially of finite type over $\bC$ and is normal and $\bQ$-Gorenstein, one always has the following formula for multiplier ideals and any $\lambda \geq 0$: 
    \[
        \mJ(R, J^{n+{\lambda}}) = J \cdot \mJ(R, J^{n+{\lambda} - 1}),
    \]
    see for instance \cite[Section 9.6]{LazarsfeldPositivity2}.  
    Indeed, we then have that 
    \[
        \overline{J^{n+\lambda-1}}\cdot \mJ(R) \subseteq \mJ(R, \overline{J^{n+\lambda-1}}) = \mJ(R, {J^{n+\lambda-1}}) = J^{\lambda} \mJ(R, J^{n-1}) \subseteq J^{\lambda}.
    \]
    The same formula holds for test ideals in $F$-finite or excellent local Noetherian domains of characteristic $p > 0$ \cite[Theorem 4.1]{HaraTakagiOnAGeneralizationOfTestIdeals}, \cf \cite[Theorem 2.1]{HaraYoshidaGeneralizationOfTightClosure}.  A version was proved for the $+$-test ideal in \cite[Corollary 6.7]{HaconLamarcheSchwede} when $R$ is complete, normal and $\bQ$-Gorenstein.  
    
    However, these sorts of formulations, such as:
    \[
        \tau(R) \cdot \overline{J^{n+\lambda-1}} \subseteq J^{\lambda},
    \]
    are straightforward corollaries of the closure variants.  Indeed, if $\overline{J^{n+\lambda-1}} \subseteq (J^{\lambda})^*$, then since $\tau(R) \cdot I^* \subseteq I$ for any $I$, we obtain $\tau(R)  \cdot \overline{J^{n+\lambda-1}} \subseteq J^{\lambda}$.
    For more discussion on variants of test ideals associated to BCM algebras and these sorts of closure properties as well as their relation to the associated test ideal theory, see \cite[Corollary 3.3.3]{DattaTuckerOpenness} and \cite{PerezRG}.
\end{remark}



\bibliographystyle{skalpha}
\bibliography{main}

\newcommand{\etalchar}[1]{$^{#1}$}
\def\cfudot#1{\ifmmode\setbox7\hbox{$\accent"5E#1$}\else
  \setbox7\hbox{\accent"5E#1}\penalty 10000\relax\fi\raise 1\ht7
  \hbox{\raise.1ex\hbox to 1\wd7{\hss.\hss}}\penalty 10000 \hskip-1\wd7\penalty
  10000\box7}
\providecommand{\bysame}{\leavevmode\hbox to3em{\hrulefill}\thinspace}
\providecommand{\MR}{\relax\ifhmode\unskip\space\fi MR}
\providecommand{\MRhref}[2]{%
  \href{http://www.ams.org/mathscinet-getitem?mr=#1}{#2}
}
\providecommand{\href}[2]{#2}
\begin{thebibliography}{BMP{\etalchar{+}}20}

\bibitem[And20]{AndreWeaklyFunctorialBigCM}
{\sc Y.~Andr\'{e}}: \emph{Weak functoriality of {C}ohen-{M}acaulay algebras},
  J. Amer. Math. Soc. \textbf{33} (2020), no.~2, 363--380. {\sf\scriptsize
  4073864}

\bibitem[AS07]{AschenbrennerSchoutensLefschetzExtensionsBigCM}
{\sc M.~Aschenbrenner and H.~Schoutens}: \emph{Lefschetz extensions, tight
  closure and big {C}ohen-{M}acaulay algebras}, Israel J. Math. \textbf{161}
  (2007), 221--310. {\sf\scriptsize 2350164}

\bibitem[Bha20]{BhattAbsoluteIntegralClosure}
{\sc B.~Bhatt}: \emph{Cohen-{M}acaulayness of absolute integral closures},
  arXiv:2008.08070.

\bibitem[BMP{\etalchar{+}}20]{BMPSTWW-MMP}
{\sc B.~Bhatt, L.~Ma, Z.~Patakfalvi, K.~Schwede, K.~Tucker, J.~Waldron, and
  J.~Witaszek}: \emph{Globally +-regular varieties and the minimal model
  program for threefolds in mixed characteristic}.

\bibitem[Bre03]{BrennerRescueSolid}
{\sc H.~Brenner}: \emph{How to rescue solid closure}, J. Algebra \textbf{265}
  (2003), no.~2, 579--605. {\sf\scriptsize 1987018}

\bibitem[BM10]{BrennerMonsky}
{\sc H.~Brenner and P.~Monsky}: \emph{Tight closure does not commute with
  localization}, Ann. of Math. (2) \textbf{171} (2010), no.~1, 571--588.
  {\sf\scriptsize 2630050 (2011d:13005)}

\bibitem[BH93]{BrunsHerzog}
{\sc W.~Bruns and J.~Herzog}: \emph{Cohen-{M}acaulay rings}, Cambridge Studies
  in Advanced Mathematics, vol.~39, Cambridge University Press, Cambridge,
  1993. {\sf\scriptsize MR1251956 (95h:13020)}

\bibitem[DT23]{DattaTuckerOpenness}
{\sc R.~Datta and K.~Tucker}: \emph{Openness of splinter loci in prime
  characteristic}, J. Algebra \textbf{629} (2023), 307--357. {\sf\scriptsize
  4583731}

\bibitem[Die10]{DietzACharacterizationOfClosureOpsInducingBCM}
{\sc G.~D. Dietz}: \emph{A characterization of closure operations that induce
  big {C}ohen-{M}acaulay modules}, Proc. Amer. Math. Soc. \textbf{138} (2010),
  no.~11, 3849--3862. {\sf\scriptsize 2679608}

\bibitem[DRG19]{DietzRGBCMViaUltraEqualChar0}
{\sc G.~D. Dietz and R.~R.~G.}: \emph{Big {C}ohen-{M}acaulay and seed algebras
  in equal characteristic zero via ultraproducts}, J. Commut. Algebra
  \textbf{11} (2019), no.~4, 511--533. {\sf\scriptsize 4039980}

\bibitem[Gab18]{GabberMSRINotes}
{\sc O.~Gabber}: \emph{Observations made after the {MSRI} workshop ``{H}ot
  {T}opics: {T}he {H}omological {C}onjectures {M}arch 12, 2018 - {M}arch 16,
  2018''},
  \url{https://www.msri.org/workshops/842/schedules/23854/documents/3322/assets/31362},
  2018.

\bibitem[HLS22]{HaconLamarcheSchwede}
{\sc C.~Hacon, A.~Lamarche, and K.~Schwede}: \emph{Global generation of test
  ideals in mixed characteristic and applications}, arXiv:2106.14329, to appear
  in Algebr. Geom.

\bibitem[HT04]{HaraTakagiOnAGeneralizationOfTestIdeals}
{\sc N.~Hara and S.~Takagi}: \emph{On a generalization of test ideals}, Nagoya
  Math. J. \textbf{175} (2004), 59--74. {\sf\scriptsize MR2085311
  (2005g:13009)}

\bibitem[HY03]{HaraYoshidaGeneralizationOfTightClosure}
{\sc N.~Hara and K.-I. Yoshida}: \emph{A generalization of tight closure and
  multiplier ideals}, Trans. Amer. Math. Soc. \textbf{355} (2003), no.~8,
  3143--3174 (electronic). {\sf\scriptsize MR1974679 (2004i:13003)}

\bibitem[HM18]{HeitmannMaBigCohenMacaulayAlgebraVanishingofTor}
{\sc R.~Heitmann and L.~Ma}: \emph{Big {C}ohen-{M}acaulay algebras and the
  vanishing conjecture for maps of {T}or in mixed characteristic}, Algebra
  Number Theory \textbf{12} (2018), no.~7, 1659--1674. {\sf\scriptsize 3871506}

\bibitem[Hei97]{HeitmannPlusClosureMixedChar}
{\sc R.~C. Heitmann}: \emph{The plus closure in mixed characteristic}, J.
  Algebra \textbf{193} (1997), no.~2, 688--708. {\sf\scriptsize 1458810}

\bibitem[Hei01]{HeitmannExtensionsOfPlusClosure}
{\sc R.~C. Heitmann}: \emph{Extensions of plus closure}, J. Algebra
  \textbf{238} (2001), no.~2, 801--826. {\sf\scriptsize 1823785}

\bibitem[Hoc75]{HochsterTopicsInTheHomologicalTheory}
{\sc M.~Hochster}: \emph{Topics in the homological theory of modules over
  commutative rings}, Published for the Conference Board of the Mathematical
  Sciences by the American Mathematical Society, Providence, R.I., 1975,
  Expository lectures from the CBMS Regional Conference held at the University
  of Nebraska, Lincoln, Neb., June 24--28, 1974, Conference Board of the
  Mathematical Sciences Regional Conference Series in Mathematics, No. 24.
  {\sf\scriptsize 0371879 (51 \#8096)}

\bibitem[Hoc94]{HochsterSolidClosure}
{\sc M.~Hochster}: \emph{Solid closure}, Commutative algebra: syzygies,
  multiplicities, and birational algebra ({S}outh {H}adley, {MA}, 1992),
  Contemp. Math., vol. 159, Amer. Math. Soc., Providence, RI, 1994,
  pp.~103--172. {\sf\scriptsize 1266182}

\bibitem[HH90]{HochsterHunekeTC1}
{\sc M.~Hochster and C.~Huneke}: \emph{Tight closure, invariant theory, and the
  {B}rian\c con-{S}koda theorem}, J. Amer. Math. Soc. \textbf{3} (1990), no.~1,
  31--116. {\sf\scriptsize MR1017784 (91g:13010)}

\bibitem[HH92]{HochsterHunekeInfiniteIntegralExtensionsAndBigCM}
{\sc M.~Hochster and C.~Huneke}: \emph{Infinite integral extensions and big
  {C}ohen-{M}acaulay algebras}, Ann. of Math. (2) \textbf{135} (1992), no.~1,
  53--89. {\sf\scriptsize 1147957 (92m:13023)}

\bibitem[HH95]{HochsterHunekeApplicationsofBigCM}
{\sc M.~Hochster and C.~Huneke}: \emph{Applications of the existence of big
  {C}ohen-{M}acaulay algebras}, Adv. Math. \textbf{113} (1995), no.~1, 45--117.
  {\sf\scriptsize 1332808}

\bibitem[HH06]{HochsterHunekeTightClosureInEqualCharactersticZero}
{\sc M.~Hochster and C.~Huneke}: \emph{Tight closure in equal characteristic
  zero}, A preprint of a manuscript, 2006.

\bibitem[HV04]{HochsterVelezDiamondClosure}
{\sc M.~Hochster and J.~D. V\'{e}lez}: \emph{Diamond closure}, Algebra,
  arithmetic and geometry with applications ({W}est {L}afayette, {IN}, 2000),
  Springer, Berlin, 2004, pp.~511--523. {\sf\scriptsize 2037107}

\bibitem[Hun92]{HunekeUniformBounds}
{\sc C.~Huneke}: \emph{Uniform bounds in {N}oetherian rings}, Invent. Math.
  \textbf{107} (1992), no.~1, 203--223. {\sf\scriptsize 1135470}

\bibitem[HL07]{HunekeLyubeznikAbsoluteIntegralClosure}
{\sc C.~Huneke and G.~Lyubeznik}: \emph{Absolute integral closure in positive
  characteristic}, Adv. Math. \textbf{210} (2007), no.~2, 498--504.
  {\sf\scriptsize 2303230 (2008d:13005)}

\bibitem[HS06]{HunekeSwansonIntegralClosure}
{\sc C.~Huneke and I.~Swanson}: \emph{Integral closure of ideals, rings, and
  modules}, London Mathematical Society Lecture Note Series, vol. 336,
  Cambridge University Press, Cambridge, 2006. {\sf\scriptsize MR2266432
  (2008m:13013)}

\bibitem[Kov17]{KovacsSerresConditionVanishingMathOverlow}
{\sc S.~Kov\'acs}: \emph{Why does the ({S2}) property of a ring correspond to
  the {H}artogs phenomenon?}, URL:https://mathoverflow.net/q/45616 (version:
  2017-04-13), 2017.

\bibitem[Laz04]{LazarsfeldPositivity2}
{\sc R.~Lazarsfeld}: \emph{Positivity in algebraic geometry. {II}}, Ergebnisse
  der Mathematik und ihrer Grenzgebiete. 3. Folge. A Series of Modern Surveys
  in Mathematics [Results in Mathematics and Related Areas. 3rd Series. A
  Series of Modern Surveys in Mathematics], vol.~49, Springer-Verlag, Berlin,
  2004, Positivity for vector bundles, and multiplier ideals. {\sf\scriptsize
  MR2095472 (2005k:14001b)}

\bibitem[Lip93]{LipmanAdjointsAndPolarsOfIdeals}
{\sc J.~Lipman}: \emph{Adjoints and polars of simple complete ideals in
  two-dimensional regular local rings}, Bull. Soc. Math. Belg. S\'er. A
  \textbf{45} (1993), no.~1-2, 223--244, Third Week on Algebra and Algebraic
  Geometry (SAGA III) (Puerto de la Cruz, 1992). {\sf\scriptsize MR1316244
  (97a:13030)}

\bibitem[Lip94]{LipmanCohenMacaulaynessInGradedAlgebras}
{\sc J.~Lipman}: \emph{Cohen-{M}acaulayness in graded algebras}, Math. Res.
  Lett. \textbf{1} (1994), no.~2, 149--157. {\sf\scriptsize 1266753}

\bibitem[LS81]{LipmanSatheyeJacobianIdealsAndATheoremOfBS}
{\sc J.~Lipman and A.~Sathaye}: \emph{Jacobian ideals and a theorem of
  {B}rian\c{c}on-{S}koda}, Michigan Math. J. \textbf{28} (1981), no.~2,
  199--222. {\sf\scriptsize 616270}

\bibitem[LT81]{LipmanTeissierPseudoRational}
{\sc J.~Lipman and B.~Teissier}: \emph{Pseudorational local rings and a theorem
  of {B}rian\c con-{S}koda about integral closures of ideals}, Michigan Math.
  J. \textbf{28} (1981), no.~1, 97--116. {\sf\scriptsize MR600418 (82f:14004)}

\bibitem[Ma18]{MaThevanishingconjectureformapsofTorandderivedsplinters}
{\sc L.~Ma}: \emph{The vanishing conjecture for maps of {T}or and derived
  splinters}, J. Eur. Math. Soc. (JEMS) \textbf{20} (2018), no.~2, 315--338.
  {\sf\scriptsize 3760297}

\bibitem[MS21]{MaSchwedeSingularitiesMixedCharBCM}
{\sc L.~Ma and K.~Schwede}: \emph{Singularities in mixed characteristic via
  perfectoid big {C}ohen-{M}acaulay algebras}, Duke Math. J. \textbf{170}
  (2021), no.~13, 2815--2890.

\bibitem[Mur21]{MurayamaSymbolicTestIdeal}
{\sc T.~Murayama}: \emph{Uniform bounds on symbolic powers in regular rings},
  arXiv:2111.06049.

\bibitem[PRG21]{PerezRG}
{\sc F.~P\'{e}rez and R.~R.~G.}: \emph{Characteristic-free test ideals}, 2021.
  {\sf\scriptsize 4312323}

\bibitem[RG18]{RGClosureOperationsThatInducedBCM}
{\sc R.~R.~G.}: \emph{Closure operations that induce big {C}ohen-{M}acaulay
  algebras}, J. Pure Appl. Algebra \textbf{222} (2018), no.~7, 1878--1897.
  {\sf\scriptsize 3763288}

\bibitem[SdS87]{SanchodeSalasBlowingupmorphismswithCohenMacaulayassociatedgradedrings}
{\sc J.~B. Sancho~de Salas}: \emph{Blowing-up morphisms with {C}ohen-{M}acaulay
  associated graded rings}, G\'{e}om\'{e}trie alg\'{e}brique et applications, I
  (La R\'{a}bida, 1984), Travaux en Cours, vol.~22, Hermann, Paris, 1987,
  pp.~201--209.

\bibitem[Sch04]{SchoutensCanonicalBCMAlgebrasAndRational}
{\sc H.~Schoutens}: \emph{Canonical big {C}ohen-{M}acaulay algebras and
  rational singularities}, Illinois J. Math. \textbf{48} (2004), no.~1,
  131--150. {\sf\scriptsize 2048219}

\bibitem[Sha81]{SharpCMProperttiesForBalancedBCM}
{\sc R.~Y. Sharp}: \emph{Cohen-{M}acaulay properties for balanced big
  {C}ohen-{M}acaulay modules}, Math. Proc. Cambridge Philos. Soc. \textbf{90}
  (1981), no.~2, 229--238. {\sf\scriptsize 620732}

\bibitem[SB74]{BrianconSkoda}
{\sc H.~Skoda and J.~Brian\c{c}on}: \emph{Sur la cl\^{o}ture int\'{e}grale d'un
  id\'{e}al de germes de fonctions holomorphes en un point de {${\bf
  C}\sp{n}$}}, C. R. Acad. Sci. Paris S\'{e}r. A \textbf{278} (1974), 949--951.
  {\sf\scriptsize 340642}

\bibitem[Smi94]{SmithTightClosureParameter}
{\sc K.~E. Smith}: \emph{Tight closure of parameter ideals}, Invent. Math.
  \textbf{115} (1994), no.~1, 41--60. {\sf\scriptsize MR1248078 (94k:13006)}

\end{thebibliography}

\end{document}